\documentclass{amsart} 

\newtheorem{thm}{Theorem}[section]

\newtheorem{prop}[thm]{Proposition}
\newtheorem{lem}[thm]{Lemma}

\newtheorem{cor}[thm]{Corollary}

\theoremstyle{definition}\newtheorem{defn}[thm]{Definition}

\theoremstyle{definition}\newtheorem{question}[thm]{Question}

\theoremstyle{remark}\newtheorem{remark}[thm]{Remark}

\theoremstyle{definition}

\theoremstyle{definition}\newtheorem{vsex}[thm]{Counter-example}

\def\R{{\mathbb R}}
\def\C{{\mathbb C}}
\def\Z{{ \mathbb Z}}
\def\N{{ \mathbb N}}
\def\Q{{ \mathbb Q}}

\def \mod{{\;\mbox{\rm mod}}\;}

\def \rk{{\mbox{\rm rk}}}
\def \deg{{\mbox{\rm deg}\,}}

\def \sgn{{\mbox{\rm sgn}}}

\def \air{{\vskip 12pt\noindent}\par}

\def \ib{\begin{enumerate}}
\def \ie{\end{enumerate}}

\def \det{{\mbox{\rm det}\,}}

\def\ib {\begin{enumerate}}
\def\ie {\end{enumerate}}

\def \Spr {{\rm Spec_r}}
\def \supp {{\rm supp}}
\def \p {{\mathfrak p}}

\def \diag {{\rm diag}}

\def \PI {{\bf (PNRI)\;}}
\def \GE {{\bf (GRI)}}
\def \dd {{\mathfrak d}}

\def\adots{\mathinner{\mkern2mu\raise 1pt\hbox{.}\mkern 3mu\raise
4pt\hbox{.}\mkern1mu\raise 7pt\hbox{{.}}}}

\usepackage{amssymb}

\title{On Positive Matrices which have a positive Smith Normal Form}
\author{Ronan Quarez}

\address{IRMAR (CNRS, URA 305), Universit\'e de Rennes 1, Campus de Beaulieu\\ 35042 Rennes Cedex, France} 
\email{e-mail : ronan.quarez@univ-rennes1.fr}
\date{\today} 

\keywords{} 
\subjclass[2000]{}

\begin{document}
\maketitle

\begin{abstract}
It is known that any symmetric matrix $M$ with entries in $\R[x]$ and which is positive semi-definite for any substitution of $x\in\R$, has a Smith normal form whose diagonal coefficients are constant sign polynomials in $\R[x]$. \par We generalize this result by considering a symmetric matrix $M$  with entries in a formally real principal domain $A$, we assume that $M$ is positive semi-definite for any ordering on $A$ and, under one additionnal hypothesis concerning non-real primes, we show that the Smith normal of $M$ is positive, up to association. Counterexamples are given when this last hypothesis is not satisfied.\par  We give also a partial extension of our results to the case of Dedekind domains. 
\end{abstract}

\section{Introduction}
Arising in various areas in Mathematics, there is a seminal result (we refer to \cite{Dj}) which says that any $n\times n$-symmetric matrix $M$ with entries in $\R[x]$ which is positive for any substitution of $x\in\R$ is a matricial sum of squares (it can be written $M=\sum_{i=1}^2N_iN_i^T$ where $N_i$ is a $n\times n$-matrice with entries in $\R[x]$).
The proof of this result as given in  \cite{Dj} uses, as a prerequisite, that $M$ has a smith normal form whose diagonal coefficients are polynomials in $\R[x]$ of constant sign.
In this article we are concern with this last property. \par Since any matrix with entries in a Principal Ideal Domain (PID in short) admits a Smith Normal Form, we consider the following :
\begin{question}\label{question}
Let $M$ be a symmetric square matrix with entries in a principal ring $A$. Assume that $M$ is positive semi-definite.  Are all the diagonal elements of its Smith Normal Form positive semi-definite up to association ?
\end{question}

Before giving a precise meaning to this question in an abstract setting, we may note that the answer to this question should clearly be positive whenever the matrix $M$ is diagonal, i.e. when $M$  is already given in its Smith Normal Form.\air 

If $A=\R[x]$ the ring of all polynomials in one variable over the reals, then the positivity of a matrix $M$ in $A^{n\times n}$ can be understood as the positivity when evaluated at any point $x\in\R$, i.e. $\phi_y(M)$ is positive-semi-definite (psd in short) for any evaluation ring-homomorphism $\phi_{y}:\R[x]\rightarrow\R$ which maps $p(x)$ onto $p(y)$. The natural extension is to consider what happens if we change $A=\R[x]$ with $A=k[x]$ where $k$ is any field. And more generally, when $A$ is an abstract principal ring ?\air

It appears to be quite natural to introduce the real spectrum of the ring $A$, which is the set of all couples $(\p,\leq)$ where $\p$ is a prime ideal of $A$ and $\leq$ is an ordering onto the field of fractions of $A/\p$. The real spectrum of $A$ can be described equivalently as the set of all ring-morphisms of $A$ into a real closed field. See Section \ref{Spr} for precise definitions and properties.

Then, saying that the matrix $M$ is positive semi-definite will mean that it is positive semi-definite with respect to any point $\alpha$ of $\Spr A$, the real spectrum of $A$, i.e. the matrix $\phi(M)$ is positive semi-definite for any ring-morphism $\phi:A\rightarrow R$ where $R$ is a real closed field.
This notion obviously coincides with the common notion of positivity in the case $A=\R[x]$.

\air 
Following the proof of \cite{Dj}, we answer Question \ref{question} by the affirmative for all principal rings $A$ such that any non-real irreducible can be associated to a positive non-real irreducible, a condition called $\PI$ in the following.
For instance, $\PI$ is satisfied when the real spectrum $\Spr A$ is a connected topological space.\par

The first example of principal domains are rings of number fields : they are treated in Section \ref{number_fields}. 
The other wide class of rings with interesting arithmetic properties are rings of coordinates of affine irreducible non-singular curves. Althought only few of them are principal, they all are Dedekind domains so, in secton \ref{Dedekind} we give some partial extensions of our framework to Dedekind domains. 


\section{Preliminaries}

The basic facts of this section are taken from \cite{BCR} and for some others we will refer to \cite{ABR}.

\subsection{The real spectrum of a ring}\label{Spr}
The ring $A$ admits an ordering if and only if $-1$ is not  sum of squares in $A$, we say then that $A$ is {\it formally real}. A prime ideal $\p$ of $A$ will be called {\it real} if the quotient ring $A/\p$ is formally real. For example, in an Unique Factorization Domain (UFD in short), an irreducible $p$ will be called real if it generates a real prime ideal.
\par
The real spectrum $\Spr A$ of a ring A is defined to be the set of all couples $\alpha=(\p,\leq_\alpha)$ where $\p$ is a real prime ideal of $A$ and $\leq_\alpha$ is an ordering on $A/\p$. We say that $\p$ is the support of $\alpha$ and denote it by $\p=\supp(\alpha)$. Equivalently, an element $\alpha\in\Spr A$ is given by of a morphism $\phi:A\rightarrow R$ where $R$ is a real closed field. Given such a data, $\phi^{-1}(0)=\p$ is a real prime ideal and the unique ordering on $R$ induces an ordering $\leq_\alpha$ onto $A/\p$. \par
It is then clear that $\Spr K$ can be seen as a subset of $\Spr A$
where $K$ stands for the fraction field of a domain $A$. 

\air
For a given $a\in A$, we say that $a> 0$ (resp. $a\geq 0$) if for all $\alpha\in\Spr A$, $a>_\alpha0$ (resp. $a\geq_\alpha 0$). Moreover, we note $\sgn[a](\alpha)=+1$ (respectively $\sgn[a](\alpha)=-1$, $\sgn[a](\alpha)=0$) if $a>_\alpha 0$ (respectively $a<_\alpha 0$, $a\in\supp(\alpha)$).\par
Now, if $M\in A^{n\times n}$ is a symmetric matrix with entries in $A$, we say that $M$ is  {\it positive-semi-definite} (psd in short) if for any morphism $\phi: A\rightarrow R$ with $R$ a real closed field, the matrix $\phi(M)$ is psd.
\air
The real spectrum of $A$ has a natural topology admitting as a basis of open subsets all the sets $(\{\alpha\in \Spr A\mid a>_\alpha0\})_{a\in A}$.

\subsection{Generizations}
We say that $\beta$ is a generization of $\alpha$ and we denote it by $\beta\rightarrow \alpha$ if $\alpha$ belongs to the closure of $\beta$. It is equivalent to saying that for all $a\in A$, if $a(\beta)\geq 0$, then $a(\alpha)\geq 0$.
\par We begin with an easy observation that will be used several time in the sequel. 
\begin{lem} Let $A$ be a UFD. 
Let also $a=p^sa'$ where $s$ is an odd integer, $p$ is a real irreducible and $p\nmid a'$ (which means that $p$ does not divide $a'$). Let $\alpha\in\Spr A$ and assume that there are two generizations $\alpha_+$ and $\alpha_-$ of $\alpha$ such that $p>_{\alpha_+}0$ and  $p<_{\alpha_-}0$. Then, $$\sgn[a](\alpha_+)\cdot\sgn[a](\alpha_-)=-1.$$
\end{lem} 
\begin{proof} Indeed, by assumption 
$\sgn[p](\alpha_+)\cdot\sgn[p](\alpha_-)=-1$ and, since $p\nmid a'$, we have  $\sgn[a'](\alpha_+)=\sgn[a'](\alpha_-)=\sgn[a'](\alpha).$
Note also that $(p)$ is prime since $A$ is UFD, and so $ a'\notin\supp(\alpha)=(p)$.
\end{proof}
We will need also the following :
\begin{lem}\label{regular_generization}
Let $p$ be an irreducible of a formally real domain $A$ such that $(p)$ is a real prime ideal of $A$. Assume that $A/(p)$ is regular. Then, for any $\alpha\in\Spr A$ whose support is $(p)$ there are two generizations $\alpha_+,\alpha_-\in\Spr(K)$ where $K$ is the fraction field of $A$. Moreover, we may take $\alpha_+,\alpha_-$ such that $p>_{\alpha_+}0$ and $p<_{\alpha_-}0$.
\end{lem}
\begin{proof}
The ring $A_{(p)}$ is a discrete valuation ring of rank  $1$. Its fraction field is $K$ and its residual field is $k$ the fraction field of the ring $A/(p)$. 
According to \cite[II.Proposition 3.3]{ABR}, any ordering $\alpha\in\Spr k$ admits at least two generizations in $\Spr K$ as wanted.
\end{proof}

\section{Unicity of the Smith Normal Form}
Let $A$ be a domain, and consider the usual equivalence relation on the set of all matrices in $A^{n\times n}$ : $M\sim N$ if there are two matrices $P,Q\in A^{n\times n}$ invertibles in $A$ ($\det P$ and $\det Q$ are units in $A$) such that $M=PNQ$.\par
Let $\diag(a_1,\ldots,a_n)$ be the diagonal matrix in $A^{n\times n}$ whose coefficients onto the diagonal are $(a_1,\ldots,a_n)$.\par
About the equivalence class of diagonal matrices, recall the well known result over a PID :

\begin{thm}
Let $A$ be a PID. Then, any matrix $M\in A^{n\times n}$ is equivalent to a diagonal matrix $D=\diag(d_1,\ldots,d_r,0,\ldots,0)$ with $d_k\mid d_{k+1}$ for all $k=1\ldots r-1$. Moreover the $d_k$'s are unique up to association.
\end{thm}
We say then that $D$  is the {\it Smith Normal Form} of the matrix $M$.\air
In fact, in this result the PID hypothesis is essential for the existence of the matrix $D$. Although, the unicity can be obtained for any domain :

\begin{prop}\label{Unicity_Smith}
Let $A$ be a domain. Assume that $D\sim D'$ where $D,D'\in A^{n\times n}$ are diagonal matrices : $D=\diag(d_1,\ldots,d_r,0,\ldots,0)$ with $d_k\mid d_{k+1}$ for all $k=1\ldots r-1$ and $D'=\diag(d'_1,\ldots,d'_s,0,\ldots,0)$ with $d'_k\mid d'_{k+1}$ for all $k=1\ldots s-1$. Then, we have $r=s$ and $(d_k)=(d'_k)$ for all $k$.
\end{prop}
We will include the proof for the convenience of the reader :
\begin{proof}
Let $K$ be the field of fractions of the domain $A$. Looking at the rank of the matrices $D$ and $D'$ viewed in $K^{n\times n}$, we get $r=\rk_K(D)=\rk_K(D')=s$. \par
For any matrix $M\in A^{n\times n}$, let us introduce $\dd_k(M)$ the ideal in $A$ generated by all minors of order $k$ of $M$. \par
\begin{lem}\label{div_minors}
Let $M=NP$ where $M,N,P\in A^{n\times n}$. Then, $\dd_k(M)\subset\dd_k(N)$.
\end{lem}
\begin{proof}
Let $\Delta$ be a $k\times k$ minor of $M=NP$, say the minor of the first $k$ rows and columns (to fix an example). Let $C_i^k$ be the truncated columns of $N$ of size $k$. Then $$\Delta=\det(p_{1,1}C_1^k+\ldots+p_{n,1}C_n^k,\ldots,p_{1,k}C_1^k+\ldots+p_{n,k}C_n^k)$$ where $P=(p_{i,j})$. Here $\Delta$ appears as a linear combination of minors of order $k$ extracted from the first $k$ lines of $N$. This implies that $\Delta$ is an element of $\dd_k(N)$. 
\end{proof}
Since the matrices $P$ and $Q$ are invertibles, by Lemma \ref{div_minors}, we have $\dd_k(D)=\dd_k(D')$ for all $k$. Now, since the matrices $D$ and $D'$ are diagonal it is easy to see that these last two ideals are in fact principal and more precisely : 
$$\dd_k(M)=(d_1\ldots d_k)\quad {\rm and}\quad \dd_k(N)=(d'_1\ldots d'_k)$$ 
Hence, we get $(d_k)=(d'_k)$ for all $k$.
\end{proof}

\section{The main results}
After  setting an abstract background, we will be able to settle our result,  following the main steps of the proof given by Djokovic in \cite{Dj}.\par
In a given ring $A$ that we may think at  as a UFD, let us introduce two conditions. \air The first one concerns the Positivity of Non-Real Irreducibles : 
\air
\PI\quad Any non-real irreducible $q$ in $A$ can be associated to a non-real irreducible which is strictly positive on all $\Spr A$.
\air
Next, we come to the second condition, relative to the Generization of a given Real Irreducible $p$ such that $(p)$ is prime : 

\air 
\GE\quad There is $\alpha\in\Spr A$ whose support is $\supp(\alpha)=(p)$ and $\beta\in\Spr A$ with support $(0)$, such that $\beta$ is a generization of $\alpha$.
\air

Note that the condition \GE  is true with respect to any real prime $(p)$ whenenever $A/(p)$ is regular (confer the proof of Lemma \ref{regular_generization} for this fact). For instance, this condition will be automatically satisfied if $A$ is a PID, since $A/(p)$ is a field in this case.\air

Now, if $A$ is a UFD, then for any irreducible $p\in A$, we may define as usually $\nu_p(a)$ to be the $p$-valuation of an element $a\in A$ to be the maximal integer $k$ such that $p^k$ divides $a$. \par Here is the main result :

\begin{thm}\label{main_pss_smith}
Let $A$ be a regular UFD, $M$ be a symmetric matrix in $A^{n\times n}$ which is positive semi-definite on $\Spr A$. 
Assume that $M$ admits a Smith Normal form, i.e. there are $d_1\vert\ldots\vert d_r$ in $A$, such that $M\sim D=\diag(d_1,\ldots,d_r,0\ldots,0)$. \par Assume furthermore that the ring $A$ satisfies the condition $\PI$ and the condition $\GE$ with respect to any real-irreducible $p$ dividing some $d_k$ and which does change of sign on $\Spr A$.\par
Then, for all $k=1\ldots r$, the element $d_k\in A$ can be associated to an element $d'_k\in A$ which is such that $d'_k>0$ everywhere on $\Spr A$. 
\end{thm}

\begin{proof}
We proceed by several reductions :

\air
$\bullet$ We may assume that $A$ is formally real ($\Spr A\not=\emptyset)$, otherwise there is nothing to do.
\air
$\bullet$ Because of Property \PI, we may assume that for all $k$, $d_k=e_kf_k$ where $e_k$ is a product of real irreducibles and $f_k>0$ on all $\Spr A$.

\air
$\bullet$ If $M=PDQ$ where $P,Q$ are invertible, we may reduce to the case where $P$ is the identity. Indeed, let $M'=DQ'$ where $Q'=Q(P^{-1})^T$ is invertible in $A^{n\times n}$ and $M'=P^{-1}M(P^{-1})^T$ remains symmetric and psd on $\Spr A$. Of course, $M$ and $M'$ have same Smith Normal Form.

\air$\bullet$ We may reduce to the case where $r=n$. Indeed, let $M=DQ$ where $Q$ is invertible, $D$ diagonal, and write $M=\left(\begin{array}{cc}M_1&M_2\\ M_3&M_4\end{array}\right)$, $Q=\left(\begin{array}{cc}Q_1&Q_2\\ Q_3&Q_4\end{array}\right)$, $D=\left(\begin{array}{cc}D_1&0\\ 0&0\end{array}\right)$ with $M_1,Q_1,D_1=\diag(d_1,\ldots,d_r)\in A^{r\times r}$ ; $M_2,Q_2\in A^{r\times (n-r)}$ ; $M_3,Q_3\in A^{(n-r)\times r}$ ; $M_4,Q_4\in A^{(n-r)\times (n-r)}$. We get then 
$$M=\left(\begin{array}{cc}D_1Q_1&D_1Q_2\\ 0&0\end{array}\right),$$ hence $M_3=M_4=0$ and by symmetry $M_2=0$. So, we are reduce to 
 $M_1=D_1Q_1$. Next, remark that $Q_1$ is necessarily invertible. Indeed, by the proof of Proposition \ref{Unicity_Smith}, we have 
 $$(\dd_r(M))=\dd_r(D)=\dd_r(D_1)=d_1\ldots d_r=(\det(M_1)),$$
which shows that $\det(Q_1)$ is invertible.

\air $\bullet$ Assume that the Theorem is not true. So there is an integer $m$ such that $d_m$ is not associated to a positive element on all $\Spr A$. 
Hence, there is a real irreducible $p$ which changes of sign on $\Spr A$ and such that $\nu_p(d_m)$ is odd. We will assume moreover that $\nu_p(d_i)$ is even for all $i\leq m$ and $\nu_p(d_1)\leq\ldots\leq \nu_p(d_m) \leq\ldots\leq   \nu_p(d_n)$.

\par

We claim now that, for all $1\leq i\leq m$ and $m\leq j\leq n$, the entry $q_{i,j}$ of the matrix $Q$ is divisible by $p$.\air

\begin{enumerate}
\item[a)] For $i=j=m$ this follows from the fact that $d_mq_{m,m}\geq 0$ on all $\Spr A$ ($d_mq_{m,m}$ is a $1\times 1$-minor of the positive matrix $DQ$). By condition \GE\, relative to $p$, there is $\alpha\in\Spr A$ with support $(p)$ and $\alpha_+,\alpha_-\in\Spr A$ two generizations of $\alpha$ with support $(0)$, such that $p>_{\alpha_+}0$ and $p<_{\alpha_-}0$. Thus, $\nu_p(q_{m,m})$ shall be odd in order to have $d_mq_{m,m}\geq_{\alpha_-} 0$. 

\item[b)] For $i=m$ and $j>m$, we check that $p\vert q_{i,j}$ is a consequence of the positivity of the following symmetric $2\times 2$-minor of  $DQ$ :
$$\left(\begin{array}{cc}
d_mq_{m,m}&d_mq_{m,j}\\
d_jq_{j,m}&d_jq_{j,j}
\end{array}\right)$$
Indeed, we have the inequality on all $\Spr A$ : 
\begin{equation}\label{eq1}
(d_md_j)(q_{m,m}q_{j,j})-(d_mq_{m,j})^2\geq 0
\end{equation}

At this point, we use the result 
\begin{lem}\label{valuation_real}
Let $A$ be a formally real UFD which is also a regular domain, and $a,b\in A$. Assume that $a-b^2\geq_\alpha 0$ for all $\alpha\in\Spr A$. Then, for all real irreducible $p$, we have 
$$\nu_p(a)\leq \nu_p(b^2)$$ 
\end{lem}
\begin{proof}
Assume that there exists $p\in A$ a real irreducible such that  $$2r+s=\nu_p(a)> \nu_p(b^2)=2r$$ with $r,s\in\N$. We write $a-b^2=p^{2r}(p^sa'-b'^2)$ with $a',b'\in A$ such that $p\nmid a'$ and $p\nmid b'$. By assumption, $p^{2r}(p^sa'-b'^2)\geq_\alpha 0$ for all $\alpha\in\Spr A$.\par Take $\alpha\in\Spr A$ such that $\supp(\alpha)=p$. By \ref{regular_generization}, there is a generization $\beta$ of $\alpha$ in $\Spr A$ such that $\supp(\beta)=(0)$.\par

By assumption, $p^{2r}(p^sa'-b'^2)\geq_\beta 0$, which yields $p^sa'-b'^2\geq_\beta 0$ since $p\notin \supp(\beta)=(0)$. By specialization, we get $p^sa'-b'^2\geq_\alpha 0$, and hence $-b'^2\geq_\alpha 0$. Necessarily, $b'\in\supp(\alpha)$, namely $p\mid b'$ : a contradiction. 
\end{proof}

Using Lemma \ref{valuation_real}, from Equation (\ref{eq1}) we get 

$$
\nu_p(d_{m})+2\nu_p(q_{m,j})\geq \nu_p(d_j)+\nu_p(q_{m,m})+\nu_p(q_{j,j})
$$\air

Since $\nu_p(d_{j})\geq \nu_p(d_{m})$ for $j>m$, it shows that

$$
2\nu_p(q_{m,j})\geq \nu_p(q_{j,j})+\nu_p(q_{m,m})
$$\air

Since $\nu_p(q_{m,m})\geq 1$, we obtain $\nu_p(q_{m,j})\geq 1$, namely $p\vert q_{m,j}$.
\air

\item[c)] For $i<m$ and $j\geq m$, we use the equality $d_iq_{i,j}=q_{j,i}d_j$ and the fact that $\nu_p(d_j)\geq \nu_p(d_m)>\nu_p(d_i)$, to conclude that $p\vert q_{i,j}$ too.
\air
\end{enumerate}

To end, we use the elementary
\begin{lem}\label{p_divides_matrix}
Let $P=(p_{i,j})\in A^{n\times n}$ be a matrix with entries in a domain $A$. Assume that there is an irreducible $p$ such that $p$ divides $m_{i,j}$ for all $1\leq i \leq r$ and $r\leq j \leq n$, with $r\in\N$.\par Then, $p$ divides $\det(P)$.
\end{lem}
\begin{proof}
We proceed by induction on $r$. If $r=1$, then the result is obvious. \par
Next, if $r>1$, we developp according to the last row and we find that $\det(P)$ is a linear combination of determinants which are all divisible by $p$ by the induction hypothesis.  
\end{proof}
By Lemma \ref{p_divides_matrix}, the irreducible $p$ divides $\det(Q)$ although $Q$ is supposed to be invertible in $A$ : a contradiction which concludes the proof.
\end{proof}

Of course, when $A$ is a Principal Ideal Domain, then $A$ is an UFD and moreover satisfies condition \GE. 
 Moreover, the Smith Normal form of a matrix always exists, so we are able to present a shorter version of Theorem \ref{main_pss_smith} under the assumption that the ring is principal. 
\begin{thm}\label{main_pss_PID}
Let $A$ be a PID and $M$ be a symmetric matrix in $A^{n\times n}$ which is positive semi-definite on $\Spr A$. Let $M\sim D=\diag(d_1,\ldots,d_r,0\ldots,0)$ with $d_1\vert\ldots\vert d_r$ in $A$ be the Smith Normal Form of $M$.
We assume furthermore that the ring $A$ satisfies the condition $\PI$. \par
Then, up to association, all the $d_k$'s are positive on $\Spr A$. 
\end{thm}


\remark\label{vs_ex}{According to the proof of Theorem \ref{main_pss_smith}, if we search for a counterexample to Question \ref{question}, we may focus on the case $n=2$. Namely, take
$$M=\left(\begin{array}{cc}
ad_1&bd_1e_1\\
bd_1e_1&cd_1e_1
\end{array}\right)=\left(\begin{array}{cc}
d_1&0\\
0&d_1e_1
\end{array}\right)\left(\begin{array}{cc}
a&be_1\\
b&c
\end{array}\right)$$
Where $$ac-b^2e_1=\epsilon,\quad \epsilon\in A^*$$

The symmetric matrix $M$ will be positive semi-definite if and ony if we have on all $\Spr A$ :  
$$\left\{\begin{array}{cc}
ad_1&\geq 0\\
cd_1e_1&\geq 0\\
d_1^2e_1&\geq 0\\
\end{array}\right.$$
So, to get a couterexample we will search for an element $d_1\in A$ which change of sign on $\Spr A$ and compatible with all the previous conditions.
} 

\endremark
\air

\section{On the conditions $\GE$ and $\PI$}
\subsection{Condition $\GE$}
We will not discuss very much this rather technical condition because it will be automatically satisfied for the class of rings we are mainly interested in. Indeed, if $A$ is principal, for any irreducible $p$, the ring $A/p$ is regular. The analogeous observation will be also valid when $A$ is a Dedekind domain (confer section \ref{Dedekind}). 

\subsection{Condition $\PI$}
We may note first that if the ring $A$ is not formally real then the condition is obviously satisfied, but Theorem \ref{main_pss_smith} has not any interest !\par
The next class of rings for which the condition is easily seen to be true is given by the following :  

\begin{prop} Condition $\PI$ is satisfied whenever the invertibles of $A$ separate the closed opens of $\Spr A$.
\end{prop}
\begin{proof}
Let $\Spr A=\cup_{i\in I}W_i$ be the decomposition of $\Spr A$ into its connected components. Let  $q$ be a non-real irreducible whose sign is $\sgn[q](W_j)=\epsilon_j=+1$ for any $j\in J$ and $\epsilon_j=-1$ for $j\in I\setminus J$. Set $P=\cup_{i\in J}W_i$ and $N=\cup_{i\in I\setminus J}W_i$, then $\Spr A=N\cup P$ is a partition of $\Spr A$ into two closed opens. Thus, by assumption there is an invertible $u\in A$ such that $u>0$ on $P$ and $u<0$ on $N$, hence $uq>0$ on all $\Spr A$.
\end{proof}

As an esay corollary, $\PI$ appears to be true for any ring whose real spectrum $\Spr A$ is connected, since in this case a non-real irreducible does not change of sign.

\air
Note moreover that 
\begin{prop}\label{localization} The condition $\PI$ is stable under localization.
\end{prop}
\begin{proof}
It suffices to see that the non-real irreducibles of $S^{-1}A$ are in one-to-one correspondance with the product of non-real irreducibles by some element of $S$, and that the correspondance which associates $\alpha\in\Spr (S^{-1}A)$ to $\alpha\in\Spr A$ such that $\supp(\alpha)\cap S=\emptyset$ is also one-to-one.  
\end{proof}

For instance, the coordinate ring of the real hyperbola $\R[x,y]/(xy-1)$ satisfies $\PI$\!. 

\remark{
Let $A$ be a principal ring. It follows from  Proposition \ref{localization} that if $A$ satisfies condition \PI, then $A_\p$ satisfies condition $\PI$ for all prime $\p$ in $A$. Beware that the converse is false. For intance, \ref{vsex1} gives a counterexample. 
}

\remark{The condition $\PI$ is closely related to the so-called {\it change of sign criterion} (see for instance \cite[Th\'eor\`eme 4.5.1]{BCR}) which says the following : \par
Let $R$ be a real closed field and $f$ an irreducible polynomial in $R[x_1,\ldots,x_n]$. Then, the ideal $(f)$ is real if and only if the polynomial $f$ changes of sign in $R^n$ : $(\exists x,y\in R^n\quad f(x)f(y)<0)$.\par
Indeed, the obvious implication of the equivalence gives condition $\PI$ for the ring $R[x_1,\ldots,x_n]$ : if  $f$ is a non-real irreducible, then $f$ does not change of sign (here the invertibles are elements in $R^*$, of constant sign).
}
\endremark

We may naturally extend this last property to any ring of polynomials over a non-necessarily real-closed field.

\begin{prop}
Let $A=k[x_1,\ldots,x_n]$ where $k$ is a formally real field. Then, the ring $A$ satisfies condition $\PI$.
\end{prop}
\begin{proof}
Start with the case of a single variable : $A=k[x]$, where $k$ is a formally real field. Let $p(x)$ be an irreducible polynomial in $k[x]$ which is non real. Up to association, we may assume that $p$ is monic. Let $\phi:A\rightarrow R$ be a ring-morphism into a real closed field $R$. Since $R$ is real closed, $p$ cannot change of sign in $R$, otherwise by continuity it would vanish on $R$ : a contradiction with the fact that $p$ is non real. Since $\lim_{x\to +\infty}\phi(p(x))=+\infty$, we get for all all morphism $\phi:A\rightarrow R$ into a real closed field $R$ and all $x\in R$, $\phi(p(x))>0$. In other words, $p>0$ on all $\Spr A$.
\par

To generalize the argument to $A=k[x_1,\ldots,x_n]$, let us order all the monomials with respect to the lexicographic ordering. Let $m(x)=\lambda x_1^{\alpha_1}\ldots x_n^{\alpha_n}$ be the higher monomial appearing in the polynomial $p(x)$. Up to association, we may assume that $\lambda=1$. Then, we look at the element $\phi(p(x))$ for a ring-morphism $A\rightarrow R$ with $R$ real closed. If we make all $x_i$'s tend to $+\infty$ such that all successive quotients $\frac{x_i}{x_{i+1}}$ tend also to $+\infty$ (i.e. $x_1\gg x_2\gg \ldots \gg x_n$), the we get
$$\phi(p(x_1,\ldots,x_n))\sim \phi(m(x_1,\ldots,x_n)).$$ Then, we conclude as previousy that $\phi(p(x_1,\ldots,x_n))>0$ for any substitution $(x_1,\ldots,x_n)\in R^n$. In other words, $p(x)>0$ on all $\Spr A$.
\end{proof}

For instance the property $\PI$ is satisfied in $\Q[x_1,\ldots,x_n]$ although the invertibles do not separate the closed opens of $\Spr A$.

\remark{
We shall mention also the link of this section with the content of \cite{Ma}. Roughly speaking, Marshall generalizes a separation result due to Schwartz in the geometric case, introducing a condition involving local $4$-elements fans. This last condition is empty in the one-dimentional geometric case, namely when $A=\R[V]$ is the ring of coordinates of an real affine plane curve. So it is possible to separate the connected components of $\Spr A$ (or equivalently those of $V$ as a variety) by polynomials. \par But, for our purpose, it does not say whether the polynomials can be taken invertible.  
} 
\endremark

In the next section, we study for which rings of number fields  the condition $\PI$ is satisfied.\air

\section{Rings of integers of number fields}\label{number_fields}

Let $K$ be a finite extension of $\Q$ of degree $n$. 
Write $K=\Q[x]/m(x)$ where $m(x)$ is an irreducible polynomial of degree $n$ over $\Q$. Denote by $a_1,\ldots,a_n$ all the roots of $m(x)$ in $\C$.
We say that $K$ is {\it totally real} if all the roots of $m(x)$ are real. A number field is totally real if and only if it can be embedded into $\R$. 

\air
Let $A$ be the ring of integers of $K$ over $\Z$. 
We define $N(a)$, the norm of an element in $A$, to be the integer $N(a)=\prod_{\phi}\phi(a)$, where $\phi$ runs the set of all the ring-homomorphisms $\phi:K\rightarrow \C$.

\begin{prop}\label{SPR_number_field}
Let $A$ be the ring of integers of a degree $n$ number field $K=\Q[x]/m(x)$. Then, $\Spr A=\Spr K$ and, as a set, it consists in $r$ points, where $r$ is the number of real roots of $m(x)$.
\end{prop}

\begin{proof}
A point of $\Spr A$ is given by a morphism $\phi:A\rightarrow R$ into a real closed field $R$. In order to describe $\Spr A$, we need as a prerequisite the classical description of the ideals in $\Z[x]$ : 

\begin{lem}\label{ideal}
Any prime ideal $\p$ of $\Z[x]$ has the form $\p=(p,f(x))$ where $p$ is a prime number in $\Z$ and $f(x)$ a polynomial in $\Z[x]$ whose reduction modulo $p$ is irreducible in $\Z/p\Z[x]$. 
\end{lem}

Now let $\p$ be a prime ideal in $A$, viewed as an ideal of $\Z[x]$ containing $m(x)$. If $p\in\p$ for a prime number $p\in\Z$, then $-1=p-1$ in $A/\p$ and $-1$ is a sum of  squares, in other word $A/\p$ is not formally real. As a consequence, any $\alpha\in\Spr A$ has support $\supp(\alpha)=(0)$ since by \ref{ideal} the ideal $\p$ shall be generated by an irreducible polynomial in $\Z[x]$ which have to divide $m(x)$. Hence $\Spr A=\Spr K$. \par
Moreover, an element of $\Spr K$ is determined by a morphism $\phi : K\rightarrow R$ where $R$ is a real closed field, hence can be identified with one root of $m(x)$.
\end{proof}

We shall note that if $K$ is Galois of degree $n$ over $\Q$, then $\Spr A$ consists in $n$ points in case $A$ is totally real, otherwise $\Spr A=\emptyset$.\par 
Start with the simplest examples of number fields :

\subsection{Quadratic number fields}
A quadratic number field has the form $K=\Q(\sqrt{d})$, where $d$ is square free in $\Z$. Recall that if $d\not\equiv 1\mod 4$, then the ring of integers of $K$ is $A=\Z[\sqrt{q}]=\Z[x]/(x^2-d)$, whereas if $d\equiv 1\mod 4$, then $A=\Z\left[\frac{1+\sqrt{q}}{2}\right]$. \air

As an application to Proposition \ref{SPR_number_field}, $\Spr A\not=\emptyset$ if and only if $d\geq 0$ and, we say in this case that $K$ is a real quadratic number field.

In summary, the real spectrum of $A=\Z[x]/(x^2-d)$ consists into two different points which can be seen as the two possible embeddings of $A$ into $\R$ : the first one given by sending $x$ onto $\sqrt{d}$ and the second one   by sending $x$ onto $-\sqrt{d}$.\air

About the units of a number quadratic field, it is well known (see for instance \cite{Co2}) that the group of units $A^*$ is isomorphic to $\Z/2\Z\times\Z$. We call $u$ a {\it fondamental unit} in $A$ if its image by the previous isomorphism can be written $(\pm 1,\pm 1)$.
\begin{prop}\label{PI_quadratic} Let $A$ be the ring of integers of a real quadratic number field. We assume that $A$ is principal. Then, $A$ satisfies conditon $\PI$ if and only if $N(u)=-1$, where $u$ is a fondamental unit in $A$.
\end{prop}
\begin{proof}
Assume that $d\not\equiv 1\mod 4$. We have $A\simeq\Z[x]/(x^2-d)$, and $\Spr A$ can be described by $\phi:x\mapsto \sqrt{d}$ and $\overline{\phi}:x\mapsto -\sqrt{d}$. Assume that $N(u)=\phi(u)\cdot\overline{\phi}(u)=-1$. Then, the unit $u$ changes of sign onto $\Spr A$, and hence separates the two points of $\Spr A$. Otherwise, it does not separate.\par
If $d\equiv 1\mod 4$, then $A\simeq\Z[x]/(4x^2-4x+1-d)$. We repeat the same argument as in the previous case, this time $\Spr A$ being described by $\phi:x\mapsto \frac{1+\sqrt{d}}{2}$ and $\overline{\phi}:x\mapsto \frac{1-\sqrt{d}}{2}$.
\end{proof}

As examples, mention that $N(u)=+1$ for $d=3,7,6,11,23,\ldots$ whereas $N(u)=-1$ for $d=2,10,26,\ldots$\air
In view of applying Theorem \ref{main_pss_PID}, we recall that the rings $\Z[\sqrt{2}]$, $\Z[\sqrt{3}]$ and $\Z[\sqrt{7}]$ are principal. And moreover, it is conjectured that there are infinitly many rings of quadratic numbers fields which are principal. 

\par
Since condition \PI is not satisfied by $A=\Z[\sqrt{3}]$ (according to \ref{PI_quadratic}), the first counterexample we give to Question \ref{question} will be the following one :
\begin{vsex}\label{vsex1}
In the ring $A=\Z[\sqrt{3}]$, $u=2+\sqrt{3}$ is a fondamental unit  which satisfies $N(u)=+1$, hence $u$ remains always positif on $\Spr A$. Consider the element $q=1+\sqrt{3}$ which obviously changes of sign on $\Spr A$. The equality $$-2=N(q)=(1+\sqrt{3})(1-\sqrt{3})$$ shows that $q$ is irreducible and moreover that it is non-real since we have $$-1\equiv 1^2\mod (1+\sqrt{3}).$$
We have futhermore the identity :
$$(1+\sqrt{3})^2=\frac{1}{2}+\left(r+\frac{\sqrt{3}}{r}\right)^2+\frac{7}{2}-\left(\frac{r^4+3}{r^2}\right)$$
where $r$ is a rational number chosed "close enough" to $\sqrt[4]{3}$ (i.e. in order that  $\frac{7}{2}-\left(\frac{r^4+3}{r^2}\right)$ is a positive rational number, and hence a sum of at most $4$ squares of rational numbers).
This last identity will furnish a counterexample to Question \ref{question}.\air

Indeed, following Remark \ref {vs_ex}, it suffices to set $\epsilon\in A_+^*$, $b=1$, $d_1=a=c=q=1+\sqrt{3}$, $e_1=\left(r+\frac{\sqrt{3}}{r}\right)^2+\frac{7}{2}-\left(\frac{r^4+3}{r^2}\right)$. So we have $e_1\geq 0$ on all $\Spr A$ (it is even a sum of squares) and $q^2=\epsilon+e_1$ with $\epsilon=\frac{1}{2}$ invertible in $A$.\par
We get the matricial equality :

$$M=
\left(\begin{array}{cc}
q&0\\
0&qe_1
\end{array}\right)\left(\begin{array}{cc}
q&e_1\\
1&q
\end{array}\right)$$
And $q$ is a non-real irreducible whose all associates always change of sign in $\Spr A$.

\end{vsex}

We may generalize all this section to any totally real number field. 

\subsection{Totaly real number fields}
Recall that $N(u)=\prod_\sigma\sigma(u)$, where $\sigma$ runs the set of all (real) embeddings $A\rightarrow\R$. 

We recall also some well known result about units of rings of integers of number fields (see for instance \cite{Co2}) :
\begin{thm}
Let $K$ be a  totally real number field of degree $n$ over $\Q$. Denote by $A$ the ring of integers of $K$ and $A^*$ the set of all units in $A$. Then,
\begin{enumerate}
\item[a)] The element $x$ is in $A^*$ if and only if $N(x)=\pm 1$ where $N(x)$ is the norm of $x
$.
\item[b)] We have the isomorphism $A/A_{{\rm tors}}\simeq \Z/2\Z$ (which is isomorphic to the group of all roots of unity in $K$).
\item[c) ] The group $A_{{\rm tors}}$ is free of rank $n-1$.
\end{enumerate}
As a consequence, $A^*\simeq \Z/2\Z\times \Z^{n-1}$.
\end{thm}

Recall that $\Spr A$ consists in $n$ distincts points which we denotes by $\alpha_1,\ldots,\alpha_{n}$. \par
Let ${\mathcal A}(\Spr A,\{-1,+1\})$ be the set of all maps from $\Spr A$ into $\{-1,+1\}$ and in order to identify the maps $f$ and $-f$ we introduce the quotient ${\mathcal A}={\mathcal A}(\Spr A,\{-1,+1\})/\{-1,+1\}$. Consider the map : 
$$\begin{array}{cccc}{\rm Sgn}:&A_{{\rm tors}}&\rightarrow& {\mathcal A}\\
&u&\mapsto& (\alpha\in\Spr A \mapsto\sgn[u]({\alpha})
\end{array}$$
An equivalent point of view would be to take for ${\mathcal A}$ the set of all functions $f$ satisfying $f(\alpha_n)=+1)$. Then, in place of the previous application Sgn, we would consider the following :
$$\begin{array}{cccc}{\rm Sgn}:&A_{{\rm tors}}&\rightarrow& {\mathcal A}(\Spr A\setminus\{\alpha_n\},\{-1,+1\})\\
&u&\mapsto& (\alpha\in\Spr A \mapsto\sgn[u]({\alpha})
\end{array}$$

Here is the generalization of \ref{PI_quadratic} :

\begin{prop}
Let $A$ be the ring of integers of a totaly real number field of degree $n$ over $\Q$. We assume that $A$ is principal.  Then, the ring $A$ satisfies the condition $\PI$ if and only if the application ${\rm Sgn}$ is an isomorphism of $\Z$-modules. i.e. 
we may choose a basis $(u_1,\ldots,u_{n-1})$ of $A_{{\rm tors}}$ such that $\sgn[u_j]({\alpha_i})=-1$ if and only if $i=j$.
\end{prop}
\begin{proof}
The ring $A$ satisfies the condition \PI if and only if for all subset $S\subset \Spr A=\{\alpha_1,\ldots,\alpha_n\}$ there exists an invertible $u$ such that $u>0$ on $S$ and $u<0$ on $\Spr A\setminus S$. This is equivalent to saying that the application ${\rm Sgn}$ is surjective. 
Since the free $\Z$-modules $A_{{\rm tors}}$ and ${\mathcal A}$ have same rank equal to $n-1$, it is an isomorphism.
\end{proof}

\air
In Theorem \ref{main_pss_PID}, we use the assumption that $A$ is principal.  This hypothesis seems to be too restrictive : indeed not all number fields are principal, neither the coordinate rings of real affine irreducible non-singular varieties. But, these two classes of rings appear to be Dedekind domains. 
It gives a motivation to search for an extension of Theorem \ref{main_pss_smith} to the class of Dedekind domains.

\section{Dedekind domains}\label{Dedekind}

\begin{defn}A domain $A$ is called {\it Dedekind} if it is noetherian, integrally closed, and if any non zero prime ideal is maximal.
\end{defn}

Roughly speaking, we may find in a Dedekind domain, the counterpart of all the arithmetic properties (for instance the existence of gcd) we have in a PID. We just have to replace the product of elements with the product of ideals. For instance, the decomposition of an element into a product of irreducible element will be replaced, in a Dedekind domain, with the decompositon of an ideal into a product of prime ideals.  \par
Note that  any Dedekind domain $A$ satisfies condition $\GE$ since $A/(p)$ is a field and hence regular for any irreducible $p$. 
\par Since all the ideals in a Dedekind domain $A$ are not necessarily principal, we shall give a counterpart for the definition of condition \PI :\air
\PI\quad Let $I=(f)$ be a principal ideal which is non real in $A$ (each associated minimal prime ideal is non-real). Then, $f$ is associated in $A$ to an element which is positive everywhere on $\Spr A$. 
\air
We may also note that the notion of Smith Normal form still exists in Dedekind domain.
Although, in general this form is not as simple as the one we have in the case of a principal ring. For instance, we may have to change the format of the matrice (see for instance \cite{Co1}). But for our purpose, we will limit oursevles to matrices which admits a diagonal Smith Normal Form. If we
 denote by $\dd_k(M)$ the ideal in $A$ generated by the $k\times k$ minors of the matrix $M$, the following results (which can be deduced from \cite{CR} for instance) give a criterion for a matrix to have a diagonal Smith Normal Form :

\begin{thm}\label{dd_smith_form}
Let $A$ be a Dedekind domain. Let $M$ and $N$ be two matrices in $A^{n\times n}$ such that $det(M)\not =0$ and $det(N)\not =0$. Then, there is $P,Q\in GL_n(A)$ such that $M=PNQ$ if and only if $\dd_k(M)=\dd_k(N)$ for all $k=1\ldots n$.
\end{thm}
If we erase the assumtion $\det(M)\cdot\det(N)\not=0$, then the result is still valid with the additional condition ${\mathfrak C}(M)={\mathfrak C}(N)$ (where ${\mathfrak C}(\cdot)$ denotes the {\it column ideal class} of a matrice).

As a consequence, 

\begin{cor}
A matrix $M$ in $A^{n\times n}$ such that $\det(M)\not =0$ admits a Smith Normal Form : $M\sim\diag(d_1,\ldots,d_n)$ with $d_1\vert\ldots\vert d_n$, $d_n\not =0$, if and only if  all the ideals $\dd_1(M),\ldots,\dd_n(M)$ are principal.
\end{cor}
\air

We now are able to formulate the counterpart of Question \ref{question} for Dedekind Domains :
\begin{question}\label{question_dd}
Let $M$ be a symmetric square matrix with entries in a formally real Dedekind Domain $A$. Assume that $M$ is positive semi-definite and admits a diagonal Smith Normal Form. Are all the diagonal elements of the Smith Normal Form positive semi-definite up to association ?
\end{question}

Here is a possible extension of Theorem \ref{main_pss_PID} which can be seen as an answer to Question \ref{question_dd}, despite the non satisfactory hypothesis about the decomposition into principal prime ideals :

\begin{thm}\label{main_dedekind_1}
Let $A$ be a Dedekind domain which satisfies the property $\PI$. Let $M$ be a symmetric matrix in $A^{n\times n}$ which we suppose te be positive semi-definite on $\Spr A$. 
Suppose that  for all $k=1\ldots n$, the ideal $\dd_k(M)$ is principal, namely $\dd_k(M)=(d_k)$, with $d_k\in A$. Suppose morever that all the primes appearing in the decomposition of $\dd_k(M)$ are principals.\par
Then, for all $k=1\ldots r$, the element $d_k\in A$ can be associated to an element $d'_k\in A$ such that $d'_k$ is positive everywhere on $\Spr A$. 
\end{thm}
\begin{proof}
Note first that the reduction to the case $r=n$ enables us to apply Theorem \ref{dd_smith_form}.\par
We follow the proof of Theorem \ref{main_pss_smith}, replacing irreducibles by prime ideals. The decomposition of $(d_k)$ into the product of its associated prime ideals (which all are principal) looks very much like the decomposition in a UFD. \par 
So $d_k=e_kf_k$ where $e_k$ is a product of some elements lying in some real prime ideals and $f_k$ is a product of some elements lying in some non-real prime ideals. Thanks to property \PI we may assume that $f_k>0$ on all $\Spr A$.
 \par
Valuations relative to irreducibles in an UFD are replaced with valuations relative to prime ideals in a Dedekind domain.
\par Note also that we have a version of Lemma \ref{regular_generization} in the Dedekind domain $A$
 : if $\p$ is real prime ideal different from $(0)$, then $\p$ is maximal and $A/\p$ is regular.
The rest of the proof follows.
\end{proof}

Another solution, if we want to get rid off the unsatisfactory assumption of the previous Theorem, is to restrict the conclusion by localization :
\begin{thm}\label{main_dedekind_2}
Let $A$ be a Dedekind domain which satisfies the property $\PI$. Let $M$ be a symmetric matrix in $A^{n\times n}$ which we suppose te be positive semi-definite on $\Spr A$. 
Suppose that  for all $k=1\ldots n$, the ideal $\dd_k(M)$ is principal, namely $\dd_k(M)=(d_k)$ with $d_k\in A$.\par
Then, for all prime ideal $\p$ in $A$ and all $k=1\ldots r$, the element $d_k$ can be associated in  $A_\p$ to an element $d'_k\in A_\p$ such that $d'_k$ is positive everywhere on $\Spr A$. 
\end{thm}
\begin{proof}
Since the ring $A$ satisfies condition $\PI$, then by localization, all the rings $A_\p$ satisfies the condition $\PI$ too. Moreover $A_p$ is a PID (confer \cite[Paragrahe 2, Th\'eor\`eme 1]{Bo}), so we may directly use  Theorem \ref{main_pss_PID}, to get  $d'_k>0$ on all $\Spr A_\p$. By specilization, we get also $d'_k>0$ on all $\Spr A$.
\end{proof}

For instance, Theorems \ref{main_dedekind_1} and \ref{main_dedekind_2} are true for the ring $A=\R[x,y]/(x^2+y^2-1)$.

\remark{
By, \cite[Paragraphe 3, Exemple 1)]{Bo}, if a Dedekind domain is UFD, then it is principal.}
\endremark

\subsection{Another counterexamples}
We state some counterexamples to Question \ref{question_dd}, all coming from the class of hyperelliptic curves. So we need to precise what the units look like in these rings :

\begin{lem}\label{hyperellipt_inver}
Let $A$ be the coordinate ring of the real affine hyperelliptic plane curve of equation $y^2-p(x)=0$, were $p(x)\in\R[x]$ has only single and real roots. If we assume that $\deg p$ is odd or if the leading coefficient of $p(x)$ is negative, then 
the set of units in $A$ is $\R^*$. 
\end{lem}
\begin{proof}
Any element $f$ of $A$ admits a unique representation of the form $a(x)y+b(x)$ where $a,b\in\R[x]$. Then, $f$ is invertible in $A$ if and only if there is $g=c(x)y+d(x)$ such that in $\R[x]$, we have  
$$\left\{\begin{array}{rcc}
b(x)d(x)+p(x)a(x)c(x)&=&1\\
a(x)d(x)+b(x)c(x)&=&0\\
\end{array}\right.$$
The first equation shows that $a$ and $b$ are coprime, so by the second we deduce that $b\vert d$ and $a\vert c$. Likewise, we get the reverve divisibility property, so $c=\alpha a$ and $d=\beta b$ with $\alpha,\beta\in\R^*$.\par
The previous system becomes
$$\left\{\begin{array}{rcc}
\beta (b(x)^2-p(x)a(x)^2)&=&1\\
(\beta+\alpha) a(x)b(x)&=&0\\
\end{array}\right.$$
The case $b(x)=0$ is impossible because of the first equality, whereas the  case $a(x)=0$ yields $b(x)\in\R^*$ as wanted.\par It remains the treat the case $a(x)b(x)\not =0$. Then, $\alpha=-\beta$ and we just note that the polynomial $b(x)^2-p(x)a(x)^2$ cannot be a constant if $\deg p(x)$ is odd or if the leading coefficient of $p(x)$ is negative.
\end{proof}

Note that if none of the conditions \ref{hyperellipt_inver} are satisfied, then it may exist in $A$ other invertibles than $\R^*$, as it is the case when $A=\R[x,y]/(y^2-(x^2-1)(x^2-2))$. For instance, the element $y+\left(x^2-\frac{3}{2}\right)$ is invertible with inverse 
$-4\left(y-\left(x^2-\frac{3}{2}\right)\right)$.

\begin{vsex}
Consider the cubic of coordinate ring $A=\R[x,y]/(y^2-x(x^2-1))$. It has two connected components which can be separated by the polynomial $q=x-\frac{1}{2}$. Note that $(q)$ is a non-real prime ideal such that $q^2=\frac{1}{4}+x^2-x$.  
As in Remark \ref {vs_ex}, to produce a counterexample, it suffices to take $\epsilon=\frac{1}{4}$, $b=1$, $d_1=a=c=q$, $e_1=x^2-x$ which is such that $e_1\geq 0$ on $\Spr A$.
\air

As another example, we may also consider the ring $B=\R[x,y]/(y^2+(x^2-1)(x^2-2))$, where the prime ideal $(x)$ in $B$ separates the two connected components of the variety. And we produce
a counterexample based upon the identity 
$$x^2=\frac{1}{3}\left(2+x^4+y^2\right)$$
Hence, we have $e_1=\frac{1}{3}\left(x^4+y^2\right)$ which is not only positive, but also a sum of squares in 
$B$.
\end{vsex}

This last argument could be repeated to any affine irreducible non-singular real plane curve which is is compact and has several connected components. 

\begin{prop}
Let $A=\R[V]$ be the coordinate ring of an affine non-singular irreducible and compact curve $V$. We assume moreover that the only units of $A$ are constants.\par Then, Question \ref{question} admits a negative answer for the ring $A$ if $V(\R)$ has at least two connected components.
\end{prop}

\begin{proof}
Assume that $\Spr A$ has at least two connected components, say $C_1$ and $C_2$. According to \cite{Ma}, we may find $a\in A$ which separates $C_1$ and $C_2$. Necessarily $(a)$ is non-real since it does not vanish on $\Spr A$. Since $V(\R)$ is compact, there is a rationnal number $r\geq 0$ such that $a^2-r> 0$ on $V(\R)$. 
By Schm\"udgen Positivestellensatz \cite{Sd}, we get that $a^2-r$ is a sum of squares in $A$. Thus, as in Remark \ref {vs_ex}, we are able to produce a counterexample to Question \ref{question_dd}.
\end{proof}


\end{document}